\documentclass[11pt]{article}

\usepackage{amsfonts,amsthm,amssymb,amsmath}
\usepackage{color,graphicx,mathrsfs}
\usepackage{hyperref}
\usepackage{cleveref}

\usepackage[top=1in,bottom=1in,left=1in,right=1in
]{geometry}
\usepackage{natbib}
\usepackage{multirow}

\usepackage{tikz}
\usetikzlibrary{calc,shapes,arrows}
\usepackage{pgfplots}
\usetikzlibrary{patterns}

\numberwithin{equation}{section}
\numberwithin{table}{section}
\numberwithin{figure}{section}

\newtheorem{assumption}{Assumption}
\newtheorem{theorem}{Theorem}
\newtheorem{lemma}{Lemma}
\newtheorem{proposition}{Proposition}

\newcommand{\R}{\mathbb{R}} 

\newcommand{\fmm}[1]{\bar{\mathcal{#1}}}
\newcommand{\fm}[1]{\bar{#1}}

\newcommand{\FM}{(\lambda,F,G)} 

\newcommand{\brackets}[1]{\left( {#1} \right)}

\usepackage[normalem]{ulem}

\title{
Convergence to Equilibrium States for Fluid Models \\
of Many-server Queues with Abandonment
}
\author{
Zhenghua Long \ and \ Jiheng Zhang\\
The Hong Kong University of Science and Technology\\
\{\href{mailto:zlong@ust.hk}{zlong},
\href{mailto:j.zhang@ust.hk}{j.zhang}\}@ust.hk
}
\date{\today}

\begin{document}

\maketitle

\begin{abstract}
Fluid models have become an important tool for the study of many-server queues with general service and patience time distributions.
The equilibrium state of a fluid model has been revealed by \cite{Whitt2006} and shown to yield reasonable approximations to the steady state of the original stochastic systems. 
However, it remains an open question whether the solution to a fluid model converges to the equilibrium state and under what condition. 
We show in this paper that the convergence holds under a mild condition. 
Our method builds on the framework of measure-valued processes developed in \cite{Zhang2013}, which keeps track of the remaining patience and service times.
\end{abstract}



\section{Introduction}

In this paper, we analyze the asymptotic behavior of fluid models for many-server queues with abandonment. We allow both the service time and patience time distributions to be general. To the best of our knowledge, \citet{Whitt2006} is the earliest to propose a fluid model for many-server queues with generally distributed service and patience times. In \cite{Whitt2006}, the equilibrium state for a fluid model is characterized and extensive simulations show that the equilibrium state of the fluid model yields reasonably good approximations to the original stochastic system in steady state. 

The challenge in studying many-server queues, especially when the service time is generally distributed, is that the status of the server pool plays an important role in the dynamics. However, describing the status itself is quite complicated. There have been two streams of work providing different modeling approaches. 
\citet{KangRamanan2010}, which is based on \cite{KaspiRamanan2011} for many-server queues without abandonment, modeled the status of the server pool by keeping track of the ``age'' (the amount of time a customer has been in service). Alternatively, \citet{Zhang2013} modeled the status of the server pool by tracking each customer's ``residual'' (the remaining service time). The fluid model proposed in \cite{KangRamanan2010} is too complicated to be analyzed. Even the existence and uniqueness of the fluid model solution is proved  using heavy traffic approximation. This paper thus builds on the second approach instead.

Both \cite{KangRamanan2010} and \cite{Zhang2013} established the fluid model as the limit of fluid-scaled stochastic processes underlying many-server queues (in Skorohod $J_1$ topology as the size of the system goes to infinity). However, the analysis of the fluid model itself remains open.  \citet{Whitt2006}, \citet{KangRamanan2010} and \citet{Zhang2013} have all been unable to show that the fluid model converges to the equilibrium states. Such a convergence was proved by
\citet{Whitt2004} for a many-server fluid model with exponentially distributed service and patience times. Taking advantage of the exponential distribution, the fluid model reduces to a one-dimensional ordinary differential equation (ODE). 

In general, proving convergence to the equilibrium states for fluid models is intrinsically difficult, even though the fluid models are just deterministic dynamic systems.
\citet{Bramson1996a,Bramson1996b} used entropy functions to prove convergence to the equilibrium states for fluid models of single-server queueing networks. Their fluid models are characterized by a set of ODEs plus some policy constraints. The challenge in these works mainly comes from the routing in the general network structure. The idea of constructing the entropy functions is based on the Kullback-Leibler distance. \citet{ZDZ2009} studied the fluid model for the limited processor sharing queue. The fluid model is essentially proved to be equivalent to an equation involving convolution with the service time distribution. Convergence of this fluid model is established based on careful analysis of the convolution equation and application of the renewal theorem. Part of the proving technique builds on \cite{PSW2006}, which studied the fluid model for overloaded processor sharing queues. 

The current work can be viewed as a sequel to \cite{Zhang2013}. We use the same definition for the fluid model, and even the same set of notations for easy connection. The modeling is close to that in \cite{ZDZ2009} but the method is significantly different due to customer abandonment (which does not appear in \cite{ZDZ2009}) and intrinsic difficulties in many-server models. \citet{PuhalskiiReed2010} offered a nice treatment for the fluid model of the many-server queue without abandonment. Though the main focus of that paper is not the fluid analysis, the elegant treatment of the fluid model helps to relax the assumption on initial customers made in \cite{Reed2009}. Abandonment, especially with a general patience time distribution, imposes significant challenges. A virtual buffer, which holds all the customers who have arrived but not yet scheduled to receive service according to the FCFS policy, is constructed to study abandonment in \cite{Zhang2013}. The idea is to keep some abandoned customers in the virtual buffer for tracking purposes. This paper adoptes the same idea. Our fluid model can be shown to be equivalent to the one in \cite{PuhalskiiReed2010} when patience time becomes infinite (no abandonment). 

We hope the analytical tools we develop in this paper can pave the way for studying more complicated many-server models such as the multiple customer class V-model studied in \cite{AGS2010a}. The analysis of any multi-class model will surely be more difficult, since the policy needs to be carefully modeled. This is out of the scope of the current paper. Having a better understanding of the basic model is a necessary step for embarking on more complicated models. 

The rest of the paper is organized as follows. Section~\ref{sec:fluid-model} introduces the fluid model with assumptions on the service and patience time distributions. Section~\ref{sec:conv-invar-stat} defines the equilibrium state and presents the main convergence result. A preliminary analysis together with some intuitions are provided in Section~\ref{sec:some-prel-analys}. Section~\ref{sec:proof-convergence} proves the main theorem based on several auxiliary tools provided in Appendix~\ref{sec:aux-lemmas}. Section~\ref{sec:conclusion} concludes this paper.

\section{Fluid Models of Many-server Queues}
\label{sec:fluid-model}

Let $\mathbb{R}$ denote the set of real numbers and $\mathbb{R}_+=[0,\infty)$. For $a,b\in\mathbb{R}$, write $a^+$ for the positive part of $a$ and $a\wedge b$ for the minimum. Denote $C_x=(x,\infty)$ and $F^c(x) = 1-F(x)$ for any distribution function. At time $t$, let $\fmm{R}(t)(C_x)$ denote the amount of fluid in the virtual buffer with remaining patience time larger than $x$. Since the virtual buffer also holds abandoned customers who have negative remaining patience times, the testing parameter $x$ is allowed to be both positive and negative for the measure $\fmm{R}(t)$. 
Introduce $\fm{R}(t)=\fmm{R}(t)(\R)$, the total fluid content in the virtual buffer.  
Denote by $\lambda$  the arrival rate. 
So at time $t$, the earliest arrived fluid content in the virtual buffer arrives at time $t-\fm{R}(t)/\lambda$.
To find out the status of the virtual buffer at time $t$, we take integral from $t-\fm{R}(t)/\lambda$ to $t$. If an infinitesimal amount of fluid content $\lambda ds$ arrives at time $s$, only a fraction $F^c(x+t-s)$ of it has remaining patience time larger than $x$ at time $t$ since $t-s$ amount of time has been spent waiting in queue. This yields the equation \eqref{eq:fluid-def-R}.
Please see Section~2 of \cite{Zhang2013} for a more detailed discussion on the virtual buffer.
Let $\fmm{Z}(t)(C_x)$ denote the amount of fluid in the server pool with remaining service time larger than $x$ at time $t$. Unlike the virtual buffer, a customer leaves the system once his remaining service time hits 0. So we restrict the testing parameter $x\in\R_+$ for the measure $\fmm{Z}(t)$. Let
\begin{equation}\label{eq:B(t)}
  \fm{B}(t)=\lambda t-\fm{R}(t).
\end{equation}
The physical intuition for $\fm{B}$ is that $\fm{B}(t)-\fm{B}(s)$ represents the amount of fluid in the virtual buffer that could  have entered service during time interval $(s,t]$. It should be pointed out that not all of it will actually enter the server pool. At time $s$, an infinitesimal amount $d\fm{B}(s)$ is scheduled to enter service after waiting in the virtual buffer for $\frac{\fm{R}(s)}{\lambda}$. Thus, a fraction $F\brackets{\frac{\fm{R}(s)}{\lambda}}$ has actually abandoned queue by time $s$. Only the rest makes it to the service. This contributes to the term $F^c\brackets{\frac{\fm{R}(s)}{\lambda}}$ in \eqref{eq:fluid-def-Z}. 
The following \emph{fluid dynamic equations} characterize how the fluid content $(\fmm{R}(t), \fmm{Z}(t))$ evolves over time.  For all $t\ge 0$,
\begin{align}
  \label{eq:fluid-def-R}
  \fmm{R}(t)(C_x) &= \lambda\int_{t-\frac{\fm{R}(t)}{\lambda}}^t F^c(x+t-s)ds, 
  \ x\in\R, \\
  \label{eq:fluid-def-Z}
  \fmm{Z}(t)(C_x) &= \fmm{Z}(0)(C_{x+t})
  +\int_0^t F^c\brackets{\frac{\fm{R}(s)}{\lambda}}
  G^c(x+t-s) d \fm{B}(s),
  \ x\in\R_+,
\end{align}
where we assume the initial state $\fmm{Z}(0)(C_x)$ is non-increasing and will converge to 0 as $x\to \infty$.
Introduce $\fm{Z}(t)=\fmm{Z}(t)(C_0)$, the fluid content in service; and $\fm{Q}(t)=\fmm{R}(t)(C_0)$, the fluid queue length.
Let $\fm{Z}(t)+\fm{Q}(t)=\fm{X}(t)$ denote the total amount of fluid in the physical system. The following non-idling constraints must be valid at any time $t\geq 0$,
\begin{align}
  \label{eq:non-idling-q}
  \fm{Q}(t) &= (\fm{X}(t)-1)^+,\\
  \label{eq:non-idling-z}
  \fm{Z}(t) &= \fm{X}(t)\wedge 1.
\end{align}

Essentially, the fluid dynamic equations \eqref{eq:fluid-def-R}--\eqref{eq:fluid-def-Z} and non-idling constraints \eqref{eq:non-idling-q}--\eqref{eq:non-idling-z} describe the measure-valued process $(\fmm{R}, \fmm{Z})$, if we regard the real-valued processes $\fm{Q},\fm{R},\fm{Z},\fm{B}$ as notations derived from the measure-valued process.
Let $\FM$ denote the \emph{fluid model} defined by \eqref{eq:fluid-def-R}--\eqref{eq:non-idling-z}.
The initial condition $(\fmm{R}(0), \fmm{Z}(0))$ is said to be \emph{valid} if it satisfies the equations \eqref{eq:fluid-def-R}--\eqref{eq:non-idling-z} at time $t=0$.

Throughout this paper, we make the following assumptions:
\begin{assumption}\label{assump:GF}
  Assume the service time distribution $G(\cdot)$ is continuous with finite mean $1/\mu$; and the patience time distribution $F(\cdot)$ is Lipschitz continuous.
\end{assumption}

According to Theorem~3.1 in \cite{Zhang2013}, under Assumption~\ref{assump:GF}, there exists a unique solution to the fluid model $\FM$ for any valid initial condition $(\fmm{R}(0), \fmm{Z}(0))$.
The theorem justifies the definition of the fluid model and provides a foundation to study the fluid model further. 
Theorem~3.3 in \cite{Zhang2013} also shows that the fluid model solution serves as the fluid limit of the many-server queueing models.

\section{Convergence to Equilibrium States}
\label{sec:conv-invar-stat}

A key property of the fluid model is that it has an equilibrium state. An equilibrium state is defined intuitively as the state from which the fluid model solution starts and remains. More precisely, $(\fmm{R}_\infty,\fmm{Z}_\infty)$ is an \emph{equilibrium state} of the fluid model $\FM$ if the solution to the fluid model with a valid initial condition $(\fmm{R}_\infty,\fmm{Z}_\infty)$ satisfies
\begin{equation*}
  (\fmm{R}(t),\fmm{Z}(t))=(\fmm{R}_\infty,\fmm{Z}_\infty),\quad
  \textrm{for all }t\ge 0.  
\end{equation*}
The equilibrium state is characterized in Theorem~3.2 in \cite{Zhang2013}. The state $(\fmm{R}_\infty,\fmm{Z}_\infty)$ is an equilibrium state of the fluid model $\FM$ if and only if it satisfies
\begin{align}
  \label{eq:eqm-B}
  \fmm{R}_\infty(C_x)&=\lambda\int_0^{\omega}F^c(x+s)ds, \ x\in \R,\\
  \label{eq:eqm-S}
  \fmm{Z}_\infty(C_x)&=\min\left(\rho,1\right)[1-G_e(x)],\ x\in \mathbb{R}_+,
\end{align}
where $\rho=\lambda/\mu$ is the traffic intensity, $w$ is a solution to the equation
\begin{equation}\label{eq:eqm-w}
  F(\omega)=\max\left(\frac{\rho-1}{\rho},0\right),
\end{equation}
and $G_e(\cdot)$ is the associated \emph{equilibrium} distribution of $G(\cdot)$ and is defined by
\begin{equation}
\label{eq:Ge}
  G_e(x)=\mu\int_0^xG^c(y)dy,\quad\textrm{for all }x\ge 0.  
\end{equation}
If equation \eqref{eq:eqm-w} has multiple solutions, then there will be multiple equilibria (any solution $w$ corresponds to an equilibrium). If the equation has a unique solution (for example when $F(\cdot)$ is strictly increasing), then the equilibrium state is unique.

The quantity $w$ is interpreted as the \emph{virtual} waiting time for an arriving customer. Roughly, in the equilibrium state, if a customer's patience time exceeds $w$, he will not abandon queue. This implies that the fraction of abandonment is $F(\omega)$, which is equal to $(\rho-1)^+/\rho$.  From \eqref{eq:eqm-B}, it is easy to see that the fluid content in the virtual buffer is
\begin{displaymath}
  \bar R_\infty = \fmm{R}_\infty(\R) = \lambda \omega,
\end{displaymath}
which is consistent with Little's law. From \eqref{eq:eqm-S}, the fluid content in the server pool is
\begin{displaymath}
  \bar Z_\infty = \fmm{Z}_\infty(C_0) = \min(\rho, 1).
\end{displaymath}
If $\rho>1$, then all servers will be busy, whereas if $\rho<1$, there will be no abandonment (on the fluid scale) and the fraction of busy servers will be $\rho$.  These observations and interpretations were first made by \cite{Whitt2006}, where approximation formulas based on a conjectured fluid model were also given, and were compared with extensive simulation results.  The approximation formulas derived from our fluid model is consistent with those in \cite{Whitt2006}.

\begin{theorem}\label{thm:con-inv}
  Under Assumption~\ref{assump:GF}, for any valid initial condition $(\fmm{R}(0), \fmm{Z}(0))$ if
  \begin{equation}
    \label{eq:initial-Z-vanish}
    \lim\limits_{x\to\infty}\fmm{Z}(0)(C_x) = 0,    
  \end{equation}
  then the fluid model solution converges to
  \begin{equation*}
    \lim_{t\to\infty}\left(\fmm{R}(t), \fmm{Z}(t)\right)
    = \left(\fmm{R}_\infty, \fmm{Z}_\infty\right).
  \end{equation*}
\end{theorem}

Note that the condition \eqref{eq:initial-Z-vanish} is mild. We do not even require the initial remaining workload in the server $\int_0^\infty\fmm{Z}(0)(C_x)dx$ to be finite.

\section{Preliminary Analysis}
\label{sec:some-prel-analys}

We now provide a preliminary analysis of the fluid model. 
Introduce two new functions
\begin{align}
  \label{eq:F_d}
  F_d(x) &= \int_0^x(1-F(y))dy,\\
  \label{eq:H(x)}
  H(x) &= 
  \begin{cases}
    F^c\brackets{F_d^{-1}\brackets{\frac{x}{\lambda}}}, 
    & \text{if }0\leq x <\lambda N_F,\\
    0, & \text{if } x\geq \lambda N_F,
  \end{cases}
\end{align}
where $N_F$ is the mean of the patience time, i.e.,
$ N_F = \int_0^\infty[1-F(y)]dy$.
Note that scaling $F_d(\cdot)$ using $N_F$ yields the equilibrium distribution of $F(\cdot)$. In this paper, we do not assume that patience time has a finite mean. Thus, $N_F$ can be either finite or infinite.  
In \cite{Zhang2013}, the following key equation is derived from the fluid equations \eqref{eq:fluid-def-R}--\eqref{eq:non-idling-z},
\begin{equation}\label{eq:X(t)}
  \fm{X}(t) = \fmm{Z}(0)(C_t)+\fm{Q}(0)G^c(t)
  +\frac{\lambda}{\mu}\int_0^t H\left((\fm{X}(t-s)-1)^+\right)dG_e(s)
  +\int_0^t(\fm{X}(t-s)-1)^+dG(s).
\end{equation}
This equation serves as the key to obtaining the existence and uniqueness of the fluid model solution. Again, it will play an important role in establishing the convergence to equilibrium states in this paper. 

To facilitate understanding as well as for technical proofs, we introduce an auxiliary process
\begin{equation}
  \label{eq:A-B}
  \fm{A}(t) = \int_0^t F^c\brackets{\frac{\fm{R}(s)}{\lambda}}d\fm{B}(s).
\end{equation}
According to the fluid dynamic equation \eqref{eq:fluid-def-R},
\begin{align}
\label{eq:R(t)Measure}
  \fmm{R}(t)(C_x) 
  = \lambda \int_0^\frac{\fm{R}(t)}{\lambda}F^c(x+s)ds.
\end{align}
Plugging $x=0$ into the above equation gives the fluid queue length
\begin{equation}
  \label{eq:Q(t)}
  \fm{Q}(t) = \lambda\int_{t-\frac{\fm{R}(t)}{\lambda}}^{t}F^c(t-s)ds=\lambda\int_0^{\frac{\fm{R}(t)}{\lambda}}F^c(s)ds.
\end{equation}
Utilizing \eqref{eq:B(t)}, \eqref{eq:H(x)} and \eqref{eq:Q(t)}, the auxiliary process can be written as
\begin{equation}\label{eq:A(t)_Q(t)}
\begin{split}
  \fm{A}(t)            
                        &=\lambda\int_0^t H(\fm{Q}(s))ds-\fm{Q}(t)+\fm{Q}(0).
\end{split}
\end{equation}
It follows from Lemma~A.3 in \citet{Zhang2013} that $\fm{A}(t)$ is non-decreasing in $t$.
On the other side, if we use the fluid dynamic equation \eqref{eq:fluid-def-Z}, then
\begin{align}
  \label{eq:Z(t)(Cx)}
  \fmm{Z}(t)(C_x)&=\fmm{Z}(0)(C_{x+t})+\int_0^t G^c(x+t-s)d \fm{A}(s),\\
  \label{eq:Z(t)(Cx)byParts}
  &=\fmm{Z}(0)(C_{x+t})+G^c(x)\fm{A}(t)-\int_0^t\fm{A}(t-s)d G(x+s),
\end{align}
where the second equality uses integration by parts.
Plugging $x=0$ into \eqref{eq:Z(t)(Cx)byParts} yields the relation between $\fm{A}$ and $\fm{Z}$
\begin{equation}\label{eq:A(t)_Z(t)}
\fm{A}(t)=\fm{Z}(t)-\fmm{Z}(0)(C_t)+\int_0^t\fm{A}(t-s)dG(s).
\end{equation}
The solution to the above renewal equation is
\begin{equation}\label{eq:A(t)Convolution}
\fm{A}(t)=\left( \fm{Z}(t)-\fmm{Z}(0)(C_t) \right)*U_G(t),
\end{equation}
where $U_G(t)=\sum_{n=0}^{\infty}G^{n*}(t)$ and $G^{n*}(t)$ is the $n$-fold convolution of $G(t)$ with itself.
The nice thing about the process $\fm{A}$ is that it connects to the process $\fm{Q}$ via the integral equation \eqref{eq:A(t)_Q(t)} and connects to the process $\fm{Z}$ via the convolution equation \eqref{eq:A(t)_Z(t)}.

\paragraph{Intuition.}
The auxiliary process $\fm{A}$ has an intuitive explanation. It can be regarded as the ``flow'' into the server pool. 
According to \eqref{eq:B(t)}, the fluid content in the virtual buffer at time $t$ is
\begin{equation*}
  \fm{R}(t) = \fm{R}(0)+ \lambda t - [\fm{B}(t)-\fm{B}(0)].
\end{equation*}
So $\fm{B}(t)-\fm{B}(0)$ can be understood as the ``flow" that has left the virtual buffer by time $t$. From \eqref{eq:A-B}, we see that the flow $\fm{A}$ is just a part of the flow $\fm{B}(t)-\fm{B}(0)$. We depict all the ``flows'' in Figure~\ref{fig:flow-diagram} for the many-server queueing model. 
Let $\fm{L}_2(t)=\fm{B}(t)-\fm{B}(0)-\fm{A}(t)$ be the difference between the aforementioned two flows. Note that $\fm{L}_2$ is not the fluid abandonment process since part of the fluid $\fm{L}_1(t)=\fm{R}(t)-\fm{Q}(t)$ in the virtual buffer has already ``abandoned'' the system (remaining patience time being less than or equal to zero). So the cumulative amount of fluid that has abandoned the system by time $t$ is
\begin{equation*}
  \fm{L}(t)=\fm{L}_1(t)+\fm{L}_2(t).
\end{equation*}
We can also write the fluid service completion process as
\begin{equation}
  \label{eq:service-completion}
  \fm{S}(t)=\fmm{Z}(0)((0,t])+\int_0^tG(t-s)d\fm{A}(s).
\end{equation}
By plugging $x=0$ into \eqref{eq:Z(t)(Cx)byParts}, one can easily verify that the balance equation
\begin{equation*}
  \fm{Z}(t) = \fm{Z}(0) + \fm{A}(t) - \fm{S}(t)
\end{equation*}
holds. 
Neither the abandonment process nor the service completion process is needed in the proof of convergence. The characterization is given for completeness and potential future applications.

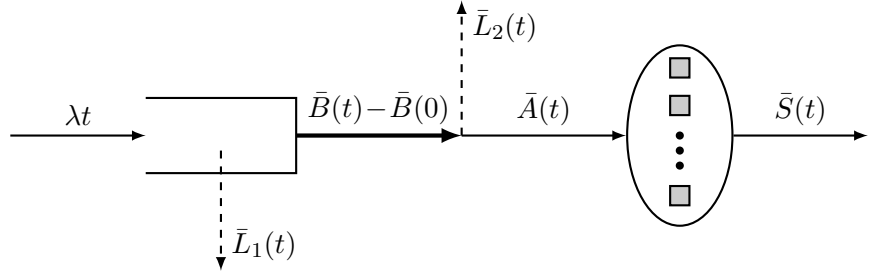
\begin{figure}[htbp]
\centering
\begin{tikzpicture}[thick,scale=1, every node/.style={scale=1}]
\tikzstyle{coor-format} = [coordinate]
\tikzstyle{dote}=[circle,draw=black, fill=black,inner sep=0.8pt]
\tikzstyle{p-format} = [ellipse, draw=black, thick,
                            minimum height=2.4cm, minimum width=1.4cm]
\tikzstyle{s-format} = [rectangle, draw=black, fill=black!20]
\def\rectanglepath{ --++(2cm,0cm)--++(0cm,1.cm) -- ++(-2cm, 0cm)}
\node [coor-format] (In0) {};
\node [coor-format,below of=In0,node distance =0.5cm] (buffer) {};
\draw (buffer) \rectanglepath;
\node [coor-format,left of=In0,node distance =1.8cm]
(Inneg) {};
\draw [draw,-latex,thick](Inneg) --node[above]{$\lambda t$} (In0);
\node [coor-format,right of=In0,node distance =2cm] (In1){};
\node [p-format, right of=In1, node distance=5.1cm] (Pool) {};

\node [coor-format,right of=In1,node distance =2.2cm]
(B0){};
\node [coor-format,label=-45:{$\bar{L}_2(t)$},above of=B0,node distance =1.8cm]
(B1){};
\draw [draw,-latex,thick,dashed](B0) --(B1);
\draw [draw,-latex, ultra thick] (In1) --node[above]{$\bar{B}(t)\!-\!\bar{B}(0)$} (B0);
\draw [draw,-latex, thick] (B0) --node[above]{$\bar{A}(t)$} (Pool);
\node [coor-format, right of=Pool, node distance=2.5cm] (Out) {};
\draw[draw,-latex,thick](Pool)--node[above]{$\bar{S}(t)$}(Out);
    \node [s-format,above of=Pool,node distance =0.9cm] (s1) {};
    \node [s-format,above of=Pool,node distance =0.4cm] (s2) {};
    \node [dote,above of=Pool,node distance =0.cm] (dote1) {};
    \node [dote,above of=Pool,node distance =-0.2cm] (dote2) {};
    \node [dote,above of=Pool,node distance =-0.4cm] (dote3) {};
    \node [s-format,below of=Pool,node distance =0.8cm] (s3) {};
\coordinate (L0) at (1,-0.2);
\node [coor-format,label=45:{$\bar{L}_1(t)$},below of=L0,node distance =1.6cm](L1){};
\draw[draw,-latex,thick,dashed](L0)--(L1);
\end{tikzpicture}
\caption{Flow Diagram of the Many-server Queue}
\label{fig:flow-diagram}
\end{figure}

\section{Proof of the Convergence}
\label{sec:proof-convergence}

Using the three auxiliary tools (Lemmas~\ref{lem:shift}--\ref{lem:comparasion}) given in the appendix, we show the convergence of the measure-valued fluid model solution to equilibrium state. The essential step is to show the convergence of the total amount of fluid $\fm{X}(t)$ to some limit, which is also of independent interest.

\begin{proposition}\label{prop:X-converge}
  Under Assumption~\ref{assump:GF}, for any valid initial condition $(\fmm{R}(0), \fmm{Z}(0))$, the fluid model solution $(\fmm{R}(t), \fmm{Z}(t))$ satisfies
  \begin{equation*}
    \lim_{t\to\infty}\fm{X}(t) = \min(\rho,1) + \lambda\int_0^\omega F^c(x)dx,
  \end{equation*}
  where $\omega$ is given in \eqref{eq:eqm-w}.
\end{proposition}
\begin{proof}
The proof is divided into three cases, the underloaded case ($\lambda<\mu$), the critically loaded case ($\lambda=\mu$) and the overloaded case ($\lambda>\mu$). 

\paragraph{Underloaded.} If $\lambda<\mu$, we just need to show that
\begin{equation}
  \label{eq:rho<1-inv}
  \lim_{t\to\infty}\fm{X}(t)=\frac{\lambda}{\mu}.
\end{equation}
By \eqref{eq:non-idling-q} and \eqref{eq:X(t)} we have,
\begin{equation*}
  \fm{Q}(t) = -\fm{Z}(t) + \fmm{Z}(0)(C_t)+\fm{Q}(0)G^c(t)
  +\frac{\lambda}{\mu}\int_0^t H(\fm{Q}(t-s))d G_e(s)+\int_0^t\fm{Q}(t-s)d G(s).
\end{equation*}
Define
\begin{equation}\label{eq:defK(t)}
  \fm{K}(t) = -\fm{Z}(t)+\fmm{Z}(0)(C_t)+\fm{Q}(0)G^c(t)+\frac{\lambda}{\mu}\int_0^t H(\fm{Q}(t-s))d G_e(s),
\end{equation}
then
\begin{equation}\label{eq:Q(t)renewal}
  \fm{Q}(t) = \fm{K}(t)+\int_0^t\fm{Q}(t-s)d G(s).
\end{equation}
So 
\begin{equation}\label{eq:Q(t)K(t)}
\fm{Q}(t)=\fm{K}(t)*U_G(t)=\int_0^t\fm{K}(t-s)d U_G(s).
\end{equation}
Consider the following two cases:\\
If $\fm{Q}(t)=0$, then by \eqref{eq:Q(t)renewal},
\begin{equation*}
  \fm{K}(t)=\fm{Q}(t)-\int_0^t\fm{Q}(t-s)d G(s)=0-\int_0^t\fm{Q}(t-s)d G(s)\leq 0.
\end{equation*}
If $\fm{Q}(t)>0$, then $\fm{Z}(t)=1$ due to non-idling constraints. Since $\lambda<\mu$ and $H(\cdot)\leq 1$, we can pick $\delta=(1-\lambda/\mu)/3$ which is positive such that $\frac{\lambda}{\mu}\int_0^tH(\fm{Q}(t-s))dG_e(s)\leq1-2\delta$. For this given $\delta>0$,  there exists a $T$ such that $\fmm{Z}(0)(C_t)+\fm{Q}(0)G^c(t)\leq \delta$ for all $t\geq T$.
It now follows from \eqref{eq:defK(t)} that for all $t\geq T$ where $\fm{Q}(t)>0$
\begin{align*}
  \fm{K}(t) &= -1+ \fmm{Z}(0)(C_t)+\fm{Q}(0)G^c(t)+\frac{\lambda}{\mu}\int_0^t H(\fm{Q}(t-s))d G_e(s)\\
            &\le -1+\delta+1-2\delta 
            = -\delta.
\end{align*}
Denote by the set $\mathcal S=\{t\ge 0: \fm{Q}(t)>0\}$ the collection of times when the fluid queue is positive. We first prove by contradiction that $m(\mathcal S)<\infty$, where $m$ is the Lebesgue measure on real numbers. Now suppose $m(\mathcal S)=\infty$.
Combining the above two cases, we have $\fm{K}(t)\leq 0$ for all $t\in[T,+\infty)$ and  $\fm{K}(t)\leq -\delta$ for all $t \in \mathcal S\cap [T,+\infty)$.  
Write the integral in \eqref{eq:Q(t)K(t)} in two parts,
\begin{equation}
  \label{eq:Q(t)bounded}
  \fm{Q}(t) = \int_0^{t-T}\fm{K}(t-s)d U_G(s)+\int_{t-T}^t \fm{K}(t-s)d U_G(s).
\end{equation}
It is clear that the first term in \eqref{eq:Q(t)bounded} is negative whenever $t>T$ since $\fm{K}(t-s)\leq 0$ for all $s\in[0,t-T]$. Moreover $\fm{K}(t-s)\leq -\delta$ for all $s\in{\mathcal S_{t,T}}$, where  $\mathcal S_{t,T}\triangleq\{s: t-s\in\mathcal S\cap [T,\infty)\}$. By the assumption that $m(\mathcal S)=\infty$, we have $m(\mathcal S_{t,T})\to\infty$ as $t\to\infty$. So the first term on the right-hand side of \eqref{eq:Q(t)bounded}
\begin{equation*}
  \int_0^{t-T}\fm{K}(t-s)dU_G(s)\leq\int_{\mathcal S_{t,T}}-\delta dU_G(s),
\end{equation*}
which converges to $-\infty$ as $t\to\infty$.
The second term on the right-hand side of \eqref{eq:Q(t)bounded} is essentially an integral on a finite interval. Let $M=\sup_{s\in[0,T]}K(s)$, then
\begin{equation*}
  \int_{t-T}^t K(t-s)dU_G(s)\leq M\int_{t-T}^tdU_G(s)\to\mu MT,
\end{equation*}
as $t\to\infty$. So we have $\lim\limits_{t\to\infty}\fm{Q}(t)=-\infty$, which contradicts $\fm{Q}(t)\geq 0$ for all $t\ge0$. Thus we have proved $m(\mathcal S)<\infty$. As a byproduct, the above analysis also yields an upper bound for the fluid queue length process. Since the first term on the right-hand side of \eqref{eq:Q(t)bounded} is always less than or equal to 0 for all $t\ge T$, and the second term has an asymptotic upper bound, there exists a constant $M_1$ such that 
\begin{equation}
  \label{eq:Q-bound-underloaded}
  \sup_{t\ge 0}\fm{Q}(t)\leq M_1+\mu M T.
\end{equation}
Since $m(\mathcal S)<\infty$, for any $\varepsilon>0$ there exists a $\tau$ such that 
$m\brackets{\mathcal S\cap [\tau,+\infty)}<\varepsilon$. 
Now consider the fluid model shifted by time $\tau$ as in Lemma~\ref{lem:shift}. 
Let $S_{\tau+t}\triangleq\{s: \tau+t-s\in\mathcal S\cap [\tau,\infty)\}$, then 
\begin{equation}
  \label{eq:tech-small-S}
  m\left(\mathcal S_{\tau+t}\right) \leq m\brackets{\mathcal{S}\cap[\tau,+\infty)}<\varepsilon.  
\end{equation}
Since $G(\cdot)$ is continuous by Assumption~\ref{assump:GF}, we can choose an $\varepsilon$ small enough and a corresponding $\tau$ such that
\begin{equation*}
  \int_0^t\fm{Q}_{\tau}(t-s)dG(s)\leq (M_1+\mu MT)\int_{\mathcal S_{\tau+t}} dG(s)
  \leq \frac{1}{2}(1-\frac{\lambda}{\mu}),
\end{equation*}
where the first inequality is due to \eqref{eq:Q-bound-underloaded} and the second one follows from \eqref{eq:tech-small-S}, Theorem~12.34 in \cite{HewittStromberg1975} and the fact that $\lambda/\mu<1$. 
Now by \eqref{eq:shift_X(t)} in Lemma~\ref{lem:shift} and that $H(\cdot)\leq1$,
\[\fm{X}_{\tau}(t)\leq \fmm{Z}(\tau)(C_t)+\fm{Q}(\tau)G^c(t)+\frac{\lambda}{\mu}+\frac{1}{2}(1-\frac{\lambda}{\mu})=\fmm{Z}(\tau)(C_t)+\fm{Q}(\tau)G^c(t)+\frac{\lambda}{2\mu}+\frac{1}{2}.\]
Therefore there exists a $\tau_1$ such that $\fm{X}(t)<1$, consequently $\fm{Q}(t)=0$, for all $t\geq\tau_1$.
Replacing $\tau$ in \eqref{eq:shift_A(t)} by $\tau_1$ gives $\fm{A}_{\tau_1}(t)=\lambda t$.
Now plugging $\tau_1$ and $x=0$ into \eqref{eq:shift_Z(t)Measure} yields 
\begin{align*}
  \fm{Z}_{\tau_1}(t)
  &= \fmm{Z}({\tau_1})(C_t)+\lambda\int_0^t G^c(t-s) ds.
\end{align*}
Replacing $(t,x)$ in \eqref{eq:Z(t)(Cx)} by $(\tau_1,t)$, it follows from the monotonicity of $\fm{A}(\cdot)$ and \eqref{eq:initial-Z-vanish} that $\fmm{Z}({\tau_1})(C_t)$ vanishes as $t\to\infty$.
The second term on the right-hand side in the above converges to $\lambda/\mu$ due to Assumption~\ref{assump:GF}. So we have proved \eqref{eq:rho<1-inv}.


\paragraph{Critically loaded.}
If $\lambda=\mu$, we only need to show that
\begin{equation}
  \label{eq:rho=1-inv}
  \lim_{t\to\infty}\fm{X}(t)=1.
\end{equation}
Same as in Lemma~\ref{lem:comparasion}, introduce a fluid model with arrival rate $\lambda_1<\lambda$ while keeping service and patience time distributions and the initial condition the same. Append the subscript $1$ to the new fluid model. 
Because $\lambda_1<\lambda=\mu$, it follows from the underloaded case that $\lim\limits_{t\to\infty}\fm{Q}_1(t)=0$. Then by \eqref{eq:A(t)_Q(t)}, we have
\begin{align*}
  \lim_{t\to\infty}\frac{\fm{A}_1(t)}{t}
  &= \lim_{t\to\infty}
  \brackets{\lambda_{1}\frac{\int_0^tH(\fm{Q}_1(s))ds}{t}-\frac{\fm{Q}_1(t)}{t}+\frac{\fm{Q}_1(0)}{t}}\\
  &=\lim_{t\to\infty}\lambda_1H(\fm{Q}_1(t))\\
  &=\lambda_1.
\end{align*}
According to Lemma~\ref{lem:comparasion}, $\fm{A}_1(t)\leq\fm{A}(t)$ for all $t\geq 0$.
So
\begin{equation*}
  \liminf_{t\to\infty}\frac{\fm{A}(t)}{t}\geq \lim_{t\to\infty}\frac{\fm{A}_1(t)}{t}=\lambda_1.
\end{equation*}
Since $\lambda_1<\lambda$ but can be arbitrarily close to $\lambda$, $\liminf\limits_{t\to\infty}\frac{\fm{A}(t)}{t}\geq \lambda$.
On the other hand, by \eqref{eq:A(t)Convolution} and the fact that $\fm{Z}(t)-\fmm{Z}(0)(C_t)\leq 1$ (due to $\fm{Z}(t)\in[0,1]$ and $\fmm{Z}(0)(C_t)\geq 0$), 
\begin{equation}
  \label{eq:tech-bound-A/t}
  \limsup_{t\to\infty}\frac{\fm{A}(t)}{t}
  =\limsup_{t\to\infty}\frac{(\fm{Z}(t)-\fmm{Z}(0)(C_t))*U_G(t)}{t}\leq\lim_{t\to\infty}\frac{U_G(t)}{t}
  =\mu.  
\end{equation}
Since $\lambda=\mu$, we must have
\begin{equation*}
  \lim_{t\to\infty}\frac{\fm{A}(t)}{t} = \lambda.
\end{equation*}
By taking limit on both sides of \eqref{eq:A(t)_Q(t)} after dividing it by $t$ and applying Lemma~\ref{le:Q(t)=o(t)},  the above limit yields
\begin{equation}
  \label{eq:tech-critical-H}
  \lim_{t\to\infty} H(\fm{Q}(t)) = 1.
\end{equation}
So for any $\varepsilon>0$, there exists a $\tau$ such that $H(\fm{Q}(t))\geq 1-\varepsilon$ for all $t>\tau$. According to \eqref{eq:shift_X(t)}, for the time-shifted fluid model we have
\begin{align*}
  \fm{X}_{\tau}(t)
  &=\fmm{Z}(\tau)(C_t)+\fm{Q}(\tau)G^c(t)+\frac{\lambda}{\mu}\int_0^t H(\fm{Q}_{\tau}(t-s))d G_e(s)+\int_0^t\fm{Q}_{\tau}(t-s)d G(s)\\
  &\geq \int_0^t H(\fm{Q}_{\tau}(t-s))d G_e(s)\\
  &\geq (1-\varepsilon)G_e(t).  
\end{align*}
This implies $\liminf\limits_{t\to\infty}\fm{X}(t)\geq 1$. 
Also by \eqref{eq:tech-critical-H} and the shape of $H(\cdot)$ in \eqref{eq:H(x)}, we must have $\lim\limits_{t\to\infty}\fm{Q}(t)=0$.
So \eqref{eq:rho=1-inv} must hold.


\paragraph{Overloaded}
If $\lambda>\mu$, we only need to show that
\begin{equation}
  \label{eq:rho>1-inv}
  \lim_{t\to\infty}\fm{X}(t)=1+\lambda\int_0^\omega F^c(x)dx,
\end{equation}
where $\omega$ is the solution to $F(\omega)=\frac{\rho-1}{\rho}$.
Introduce another fluid model with arrival rate $\lambda_2=\mu$ while keeping service and patience time distributions and the initial condition the same. Append the subscript $2$ to the new fluid model. This fluid model is actually critically loaded. So
\begin{equation*}
  \lim_{t\to\infty}\frac{\fm{A}_2(t)}{t}=\lambda_2=\mu.
\end{equation*}
It again follows from the comparison result in Lemma~\ref{lem:comparasion} that
\begin{equation*}
  \liminf_{t\to\infty}\frac{\fm{A}(t)}{t}\geq \lim_{t\to\infty}\frac{\fm{A}_2(t)}{t}=\mu.
\end{equation*}
Note that \eqref{eq:tech-bound-A/t} holds regardless of the load, so 
\begin{equation*}
  \lim_{t\to\infty}\frac{\fm{A}(t)}{t}=\mu.
\end{equation*}
Again, by taking limit on both sides of \eqref{eq:A(t)_Q(t)} after dividing it by $t$ and applying Lemma~\ref{le:Q(t)=o(t)}, the above limit yields
\begin{equation*}
  \lim_{t\to\infty} \lambda H(\fm{Q}(t)) = \mu.
\end{equation*}
Applying the form of $H(\cdot)$ in \eqref{eq:H(x)} gives
\begin{equation*}
  \lim_{t\to\infty} F\left(F_d^{-1}(\fm{Q}(t)/\lambda)\right)=\frac{\rho-1}{\rho}.
\end{equation*}
It follows from the definition of $F_d(\cdot)$ in \eqref{eq:F_d} that
\begin{equation*}
  \lim_{t\to\infty} \fm{Q}(t) = \lambda \int_0^\omega F^c(x) dx.
\end{equation*}
So \eqref{eq:rho>1-inv} is proved.
\end{proof}

Finally, we present the proof to our main result. 

\begin{proof}[Proof of Theorem~\ref{thm:con-inv}]
Since the space of real numbers is separable and $\fmm{R}_\infty(\{x\})=\fmm{Z}_\infty(\{x\})=0$ for all $x$, according to Property~(iv) of the Prohorov metric on page 72 in \cite{Billingsley1999}, it suffices to show that
\begin{align}
  \label{eq:tech-R-C_x}
  \lim_{t\to\infty}\fmm{R}(t)(C_x) &= \fmm{R}_{\infty}(C_x),\\ 
  \label{eq:tech-Z-C_x}
  \lim_{t\to\infty}\fmm{Z}(t)(C_x) &= \fmm{Z}_\infty(C_x). 
\end{align}
It follows from \Cref{prop:X-converge} that the limit of $\fm{Q}(t)$ is 
\begin{equation*}
  \label{Q(t)limits}
  \lim_{t\to\infty}\fm{Q}(t) = \fm{Q}_{\infty} 
  = \lambda\int_0^\omega F^c(x)dx, 
\end{equation*}
where $\omega$ is a solution to $F(\omega)=\max\left(\frac{\rho-1}{\rho},0\right)$.
Then by \eqref{eq:H(x)} and \eqref{eq:Q(t)},
\begin{align*}
  \lim_{t\to\infty} H(\fm{Q}(t)) 
  &= H(\fm{Q}_{\infty}) =
  \begin{cases}
    1, &\lambda\leq \mu,\\
    \mu/\lambda, &\lambda>\mu, 
  \end{cases}\\
  \lim_{t\to\infty}\fm{R}(t) 
  &= \fm{R}_{\infty} 
  = \lambda\omega.
\end{align*}
Hence, \eqref{eq:tech-R-C_x} follows from \eqref{eq:eqm-B}, \eqref{eq:R(t)Measure} and the above limit. 
Next we consider the convergence of $\fmm{Z}(t)(C_x)$.
According to the convergence of $\fm{Q}(t)$ and $H(\fm{Q}(t))$ in the above, for any $\varepsilon>0$, there exists a $\tau$ such that $|H(\fm{Q}(\tau+t))-H(\fm{Q}_\infty)|\leq \varepsilon$ and $|\fm{Q}(\tau+t)-\fm{Q}_\infty|\leq \varepsilon$ for $t>0$. 
Now consider the fluid model shifted by $\tau$. 
Plugging \eqref{eq:shift_A(t)} into \eqref{eq:shift_Z(t)Measure} yields
\begin{align*}
  \fmm{Z}_\tau(t)(C_x)
  &= \fmm{Z}(\tau)(C_{x+t})+\lambda\int_0^tH(\fm{Q}_\tau(t-s))G^c(x+s)ds\\
  &\quad -\fm{Q}_\tau(t)G^c(x)+\fm{Q}_\tau(0)G^c(x+t)+\int_0^t\fm{Q}_\tau(s)dG^c(x+t-s).
\end{align*}
On the one hand, we have
\begin{align*}
  \fmm{Z}_\tau(t)(C_x)
  &\geq \fmm{Z}(\tau)(C_{x+t})+\lambda\int_0^t(H(\fm{Q}_\infty)-\varepsilon)G^c(x+s)ds\\
  &\quad -\fm{Q}_\tau(t)G^c(x)+\fm{Q}_\tau(0)G^c(x+t)+\int_0^t(\fm{Q}_\infty-\varepsilon) dG^c(x+t-s)\\
  &= \fmm{Z}(\tau)(C_{x+t})+\lambda H(\fm{Q}_\infty)\int_0^tG^c(x+s)ds-\lambda\varepsilon\int_0^t G^c(x+s)ds\\
  &\quad -\fm{Q}_\tau(t)G^c(x)+\fm{Q}_\tau(0)G^c(x+t)+ \fm{Q}_\infty G^c(x)-\fm{Q}_\infty G^c(x+t)-\varepsilon\left(G^c(x)-G^c(x+t)\right).
\end{align*}
Since $\varepsilon>0$ can be arbitrarily small, the above inequality implies 
\begin{equation*}
  \liminf_{t\to\infty}\fmm{Z}(t)(C_x)\geq \lambda H(\fm{Q}_\infty)\int_0^\infty G^c(x+s)ds.
\end{equation*}
On the other hand, we can show in a similar way that
\begin{equation*}
  \limsup_{t\to\infty}\fmm{Z}(t)(C_x)\leq \lambda H(\fm{Q}_\infty)\int_0^\infty G^c(x+s)ds.
\end{equation*}
Thus we have
\begin{equation*}
  \label{Z(t)MeasureLimit}
  \lim_{t\to\infty}\fmm{Z}(t)(C_x)
  =\frac{\lambda H(\fm{Q}_\infty)}{\mu}\left( 1-\mu\int_0^x G^c(s)ds \right).
\end{equation*}
So \eqref{eq:tech-Z-C_x} follows from \eqref{eq:eqm-S} and \eqref{eq:Ge}.
\end{proof}

\section{Conclusion}
\label{sec:conclusion}

The proof is made possible by carefully analyzing the fluid dynamic equations \eqref{eq:fluid-def-R}--\eqref{eq:fluid-def-Z} and some auxiliary processes. 
The most important step is to show the convergence of total amount of the fluid process $\fm{X}(t)$. However, the measure-valued processes play a significant role in analyzing the real-valued process $\fm{X}(t)$. This paper contributes to the understanding of fluid models of many-server queues with general service and patience times. We believe the analytical tools we developed for the fluid models can contribute to the study of many-server queues with multiple customer classes and server pools, where a dynamic control policy plays a prominent role. 

\appendix

\section{Some Auxiliary Lemmas}
\label{sec:aux-lemmas}

\paragraph{Time shift of the fluid model.}
For any $\tau\geq 0$, denote $\left(\fmm{Z}_\tau(t),\fmm{R}_\tau(t)\right)=\left(\fmm{Z}(\tau+t),\fmm{R}(\tau+t)\right)$. The time shift for all the derived ``status'' quantities such as $\fm{Q}_\tau(\cdot)$, $\fm{R}_\tau(\cdot)$, $\fm{Z}_\tau(\cdot)$ and $\fm{X}_\tau(\cdot)$ are defined in the same way, e.g., $\fm{Q}_\tau(t)=\fm{Q}(\tau+t)$. However, let $\fm{A}_\tau(t)=\fm{A}(\tau+t)-\fm{A}(\tau)$ since $\fm{A}(t)$ records the ``cumulative'' amount of fluid that has entered service by time $t$. The time shift $\fm{B}_\tau(\cdot)$ is defined in the same way as that for the cumulative process $\fm{A}(\cdot)$. The following lemma says that the time-shifted fluid model can be regarded as ``restarting'' from time $\tau$ by viewing $\left(\fmm{Z}(\tau),\fmm{R}(\tau)\right)$ as the initial condition.  In particular, we show the time-shifted versions of \eqref{eq:X(t)}, \eqref{eq:A(t)_Q(t)} and \eqref{eq:Z(t)(Cx)}, which will be used in the proof of convergence.

\begin{lemma}\label{lem:shift}
Any time-shifted fluid model solution $\left(\fmm{Z}_\tau(t),\fmm{R}_\tau(t)\right)$ satisfies
\begin{align}\label{eq:shift_X(t)}
  &\fm{X}_{\tau}(t) = 
  \fmm{Z}(\tau)(C_t)+\fm{Q}(\tau)G^c(t)+\frac{\lambda}{\mu}
  \int_0^t H(\fm{Q}_{\tau}(t-s))d G_e(s)+\int_0^t\fm{Q}_{\tau}(t-s)d G(s),\\
  \label{eq:shift_A(t)}
  &\fm{A}_\tau(t) =\lambda\int_0^tH(\fm{Q}_\tau(s))ds-\fm{Q}_\tau(t)+\fm{Q}_\tau(0),\\
  \label{eq:shift_Z(t)Measure}
  &\fmm{Z}_\tau(t)(C_x) =\fmm{Z}(\tau)(C_{x+t})+\int_0^tG^c(x+t-s)d\fm{A}_\tau(s).
\end{align}
\end{lemma}
\begin{proof}
Plugging \eqref{eq:A(t)_Q(t)} into \eqref{eq:Z(t)(Cx)} and applying integration-by-parts give
\begin{align*}
  \fmm{Z}(\tau)(C_t)
  &= \fmm{Z}(0)(C_{\tau+t})+\lambda\int_0^\tau G^c(\tau+t-s)H(\fm{Q}(s))ds-\int_0^\tau G^c(\tau+t-s)d\fm{Q}(s)\\
  &= \fmm{Z}(0)(C_{\tau+t})+\frac{\lambda}{\mu}\int_t^{\tau+t}H(\fm{Q}(\tau+t-s))dG_e(s)\\
  &\quad -\fm{Q}(\tau)G^c(t)+\fm{Q}(0)G^c(\tau+t)+\int_{t}^{\tau+t}\fm{Q}(\tau+t-s)dG(s).
\end{align*}
So the right-hand side of \eqref{eq:shift_X(t)} becomes
\begin{align*}
  \fm{Z}(0)(C_{\tau+t})+\fm{Q}(0)G^c(\tau+t)
  +\frac{\lambda}{\mu}\int_0^{\tau+t}H(\fm{Q}(\tau+t-s))dG_e(s)
  +\int_0^{\tau+t}\fm{Q}(\tau+t-s)dG(s),
\end{align*}
which equals $\fm{X}(\tau+t)$ by \eqref{eq:X(t)}. 
Thus \eqref{eq:shift_X(t)} follows by applying the time shift definition.

The equation \eqref{eq:shift_A(t)} follows easily by the time shift definition and \eqref{eq:A(t)_Q(t)},
\begin{align*}
  \fm{A}_\tau(t)
  &= \lambda\int_\tau^{\tau+t}H(\fm{Q}(s))ds-\fm{Q}(\tau+t)+\fm{Q}(\tau)\\
  &= \lambda\int_0^tH(\fm{Q}_\tau(s))ds-\fm{Q}_\tau(t)+\fm{Q}_\tau(0).
\end{align*}
It follows from \eqref{eq:Z(t)(Cx)} and the time shift definition that the right-hand side of \eqref{eq:shift_Z(t)Measure} becomes
\begin{align*}
  &\quad\fmm{Z}(0)(C_{x+\tau+t})+\int_0^{\tau}G^c(x+\tau+t-s)d\fm{A}(s)+\int_0^tG^c(x+t-s)d\fm{A}(\tau+s)\\
  &= \fmm{Z}(0)(C_{x+\tau+t})+\int_0^{\tau}G^c(x+\tau+t-s)d\fm{A}(s)+\int_\tau^{\tau+t}G^c(x+\tau+t-s)d\fm{A}(s)\\
  &= \fmm{Z}(0)(C_{x+\tau+t})+\int_0^{\tau+t}G^c(x+\tau+t-s)d\fm{A}(s).
\end{align*}
By  \eqref{eq:Z(t)(Cx)}, the above equality is equal to $\fmm{Z}(\tau+t)(C_x)=\fmm{Z}_\tau(t)(C_x)$. So \eqref{eq:shift_Z(t)Measure} is proved.
\end{proof}

\paragraph{An asymptotic bound on the growth of the fluid queue length.}
Recall that $N_F$, the mean of patience time distribution $F(\cdot)$, can be finite or infinite. If $N_F<\infty$, then it follows from \eqref{eq:Q(t)} that
\[\fm{Q}(t)\leq \lambda \int_0^{+\infty}F^c(s)ds=\lambda N_F<\infty.\]
So the fluid queue length is uniformly bounded by the constant $\lambda N_F$ in this case. 
However, $N_F$ is also allowed to be infinite in our study. The fluid queue length does not have an upper bound any more. We show the following bound on the growth of the fluid queue length, which will help to prove our main result of convergence. 

\begin{lemma}\label{le:Q(t)=o(t)}
For any valid initial condition $(\fmm{R}(0), \fmm{Z}(0))$, the fluid model solution satisfies
\begin{equation*}
  \lim_{t\to\infty}\frac{\fm{Q}(t)}{t}=0.
\end{equation*}
\end{lemma}

\begin{proof}
The result trivially holds when $N_F<\infty$. We focus on the case when $N_F=\infty$ in this proof. 
It is easy to see that $F(t)< 1$ for any $t>0$ when $N_F=\infty$. According to \eqref{eq:A-B} and the monotonicity of $\fm{A}(t)$, the process $\fm{B}(t)$ is also non-decreasing. Using the definition of $\fm{B}(t)$ in \eqref{eq:B(t)}, we have $\fm{B}(t)-\fm{B}(0)=\lambda t-\fm{R}(t)+\fm{R}(0)\geq 0$, or equivalently
\begin{equation*}
  t-\frac{\fm{R}(t)}{\lambda}\geq -\frac{\fm{R}(0)}{\lambda}.
\end{equation*}
Applying the above inequality to \eqref{eq:Q(t)} gives
\[\begin{split}
\fm{Q}(t)=& \lambda\int_{t-\frac{\fm{R}(t)}{\lambda}}^0F^c(t-s)ds+\lambda\int_0^tF^c(t-s)ds\\
\leq & \lambda\int_{-\frac{\fm{R}(0)}{\lambda}}^0 F^c(t-s)ds+\lambda\int_0^t F^c(t-s)ds\\
=& \fm{Q}(0)+\lambda\int_0^t F^c(s)ds.
\end{split}\]
Thus
\begin{equation*}
  \lim_{t\to\infty}\frac{\fm{Q}(t)}{t}
  \leq\lim_{t\to\infty}\frac{\fm{Q}(0)+\lambda\int_0^t F^c(s)ds}{t}
  =\lim_{t\to\infty} \lambda F^c(t)=0.  
\end{equation*}
\end{proof}

\paragraph{A comparison result.}
Consider two $G/GI/N+GI$ fluid models with the same service time distribution $G$, patience time distribution $F$ and the valid initial condition $\left(\fmm{R}(0), \fmm{Z}(0)\right)$. The only difference is the arrival rate. Let $\left(\fmm{R}_i(\cdot),\fmm{Z}_i(\cdot)\right)$ denote the fluid model solution corresponding to the arrival rate $\lambda_i$, $i=1,2$. Denote by $\fm{Q}_i$, $\fm{R}_i$, $\fm{Z}_i$ and $\fm{A}_i$ the corresponding derived process associated with the $i$th fluid model. We have the following comparison result. 

\begin{lemma}\label{lem:comparasion} 
  Suppose $\lambda_1\leq\lambda_2$, then $\fm{A}_1(t)\leq\fm{A}_2(t),\fm{S}_1(t)\leq\fm{S}_2(t)$ for any $t\geq 0$.
\end{lemma}
\begin{proof}
We first prove the comparison of $\fm{A}_i$'s. For any $\delta>0$, let 
\begin{equation*}
  \tau=\inf\{t\in \R_+:\fm{A}_1(t)-\fm{A}_2(t)\geq\delta\}
\end{equation*}
be the first time when $\fm{A}_1$ exceeds $\fm{A}_2$ by $\delta$. Since the two fluid models start from the same initial condition, we must have $\tau>0$. Now the objective is to show that $\tau=\infty$. Suppose $\tau$ is finite. 
For any $t\in[0,\tau]$, if $\fm{Z}_1(t)\leq\fm{Z}_2(t)$, then by \eqref{eq:A(t)_Z(t)} 
\begin{equation}
  \label{eq:tech-logic-Z-A}
  \begin{split}
    \fm{A}_1(t)-\fm{A}_2(t)
    &=\fm{Z}_1(t)-\fm{Z}_2(t)-\int_0^t(\fm{A}_1(s)-\fm{A}_2(s))dG(t-s)\\
    &\leq -\int_0^t(\fm{A}_1(s)-\fm{A}_2(s))dG(t-s)\\
    &<\delta G(t)
    \leq \delta,
  \end{split}
\end{equation}
for any $t\in[0,\tau]$. This implies that $\fm{Z}_1(\tau)>\fm{Z}_2(\tau)$. 
A direct consequence is that 
\begin{equation}
  \label{eq:tech-Q2=0}
  \fm{Q}_2(\tau)=0,
\end{equation}
due to the non-idling equations \eqref{eq:non-idling-q} and \eqref{eq:non-idling-z}.
Let 
\[r=\sup\{t<\tau:\fm{Q}_1(t)<\fm{Q}_2(t)\}\vee 0\]
be the last time $\fm{Q}_1$ is less than $\fm{Q}_2$. 
Thus $\fm{Q}_1(t)\geq\fm{Q}_2(t)$ for each $t\in[r,\tau]$. 
Then it follows from \eqref{eq:Q(t)} that 
\begin{equation}
  \label{eq:tech-R/lambda-comp}
  \frac{\fm{R}_1(t)}{\lambda_1}\geq\frac{\fm{R}_2(t)}{\lambda_2},\quad \textrm{ for any } t\in[r,\tau].
\end{equation}
According to \eqref{eq:B(t)} and \eqref{eq:A-B},
\begin{align}
  \fm{A}_1(\tau)-\fm{A}_2(\tau)
  &=\fm{A}_1(r)-\fm{A}_2(r)-\left[
    \int_r^{\tau}F^c \left( \frac{\fm{R}_1(s)}{\lambda_1}\right)d\fm{R}_1(s)
    -\int_r^\tau F^c\left(\frac{\fm{R}_2(s)}{\lambda_2}\right)d\fm{R}_2(s)
    \right]\nonumber\\
  &\quad+\lambda_1\int_r^\tau F^c \left( \frac{\fm{R}_1(s)}{\lambda_1}\right)ds
         -\lambda_2\int_r^\tau F^c \left( \frac{\fm{R}_2(s)}{\lambda_2}\right)ds\nonumber\\
  \label{eq:ACom}
  &\leq \fm{A}_1(r)-\fm{A}_2(r)-\fm{Q}_1(\tau)+\fm{Q}_1(r)+\fm{Q}_2(\tau)-\fm{Q}_2(r),
\end{align}
where the inequality is due to \eqref{eq:tech-R/lambda-comp} and the following derivation based on \eqref{eq:Q(t)}
\begin{equation*}
  \int_r^\tau F^c\left(\frac{\fm{R}_1(s)}{\lambda_1}\right)d\fm{R}_1(s)
  =\lambda_1\int_{\frac{\fm{R}_1(r)}{\lambda_1}}^{\frac{\fm{R}_1(\tau)}{\lambda_1}}F^c(s)ds=\fm{Q}_1(\tau)-\fm{Q}_1(r).
\end{equation*}
From the definition of $r$ and the continuity of $\fm{Q}_i$, we have $\fm{Q}_1(r)=\fm{Q}_2(r)$. Following this and \eqref{eq:tech-Q2=0}, the above inequality \eqref{eq:ACom} can be continued as 
\begin{align*}
  \fm{A}_1(\tau)-\fm{A}_2(\tau)
  \le \fm{A}_1(r)-\fm{A}_2(r)-\fm{Q}_1(\tau).
\end{align*}
If $r=0$, then $\fm{A}_1(\tau)-\fm{A}_2(\tau)\le 0$ due to the same initial condition and non-negativity of $\fm{Q}_i$. 
If $r>0$, then $\fm{Z}_2(r)=1$ by the non-idling constraints \eqref{eq:non-idling-q}--\eqref{eq:non-idling-z} and continuity of $\fm{Z}_i$. So $\fm{Z}_1(r)\le \fm{Z}_2(r)$, which implies that $\fm{A}_1(r)-\fm{A}_2(r)<\delta$ by \eqref{eq:tech-logic-Z-A}. Therefore, we have $\fm{A}_1(\tau)-\fm{A}_2(\tau)<\delta$ in this case. 
Both cases contradict the definition of $\tau$. So $\tau$ cannot be finite. Thus, we have proved that $\fm{A}_1(t)\leq\fm{A}_2(t)$ for all $t\geq 0$.

The equation in  \eqref{eq:service-completion} implies
\begin{equation*}
  \fm{S}(t)=\fmm{Z}(0)((0,t))+\int_0^t\fm{A}(t-s)dG(s).
\end{equation*}
Thus, the comparison of the service completion processes $\fm{S}_i$'s follows immediately from the above equation and the comparison of $\fm{A}_i$'s.
\end{proof}

\bibliography{pub}

\end{document}